\documentclass[12pt, reqno]{amsart}
\usepackage{amsmath, amsthm, amscd, amsfonts, amssymb, graphicx, color, mathrsfs}
\usepackage[bookmarksnumbered, colorlinks, plainpages]{hyperref}
\hypersetup{colorlinks=true,linkcolor=red, anchorcolor=green, citecolor=cyan, urlcolor=red, filecolor=magenta, pdftoolbar=true}

\newtheorem{theorem}{Theorem}[section]
\newtheorem{lemma}[theorem]{Lemma}
\newtheorem{proposition}[theorem]{Proposition}
\newtheorem{corollary}[theorem]{Corollary}
\theoremstyle{definition}
\newtheorem{definition}[theorem]{Definition}
\newtheorem{example}[theorem]{Example}

\theoremstyle{remark}
\newtheorem{remark}[theorem]{Remark}
\numberwithin{equation}{section}

\newcommand{\EE}{\mathcal{E}}
\newcommand{\FF}{\mathcal{F}}

\newcommand{\LL}{\mathcal{L}}

\newcommand{\PP}{\mathcal{P}}
\newcommand{\QQ}{\mathcal{Q}}

\renewcommand{\SS}{\mathscr{S}}

\newcommand{\field}[1]{\mathbb{#1}}

\newcommand{\R}{\field{R}}
\newcommand{\N}{\field{N}}

\newcommand{\E}{\field{E}}

\renewcommand{\P}{\field{P}}

\newcommand{\al}{\alpha}

\newcommand{\ep}{\varepsilon}

\newcommand{\la}{\lambda}

\newcommand{\Si}{\Sigma}

\newcommand{\Om}{\Omega}

\newcommand{\sm}{\setminus}

\begin{document}

\setcounter{page}{1}

\title[Markov chains under nonlinear expectation]{Markov chains under nonlinear expectation}

\author[M.~Nendel]{Max Nendel}
\address{Max Nendel, Bielefeld University, Center for Mathematical Economics, Bielefeld, Germany}
\email{max.nendel@uni-bielefeld.de}


\let\thefootnote\relax\footnote{The author would like to thank Robert Denk and Michael Kupper for their valuable comments, support and suggestions related to this work. Financial support through the German Research Foundation via CRC 1283 is gratefully acknowledged.}




\date{\today}

\begin{abstract}
  In this paper, we consider continuous-time Markov chains with a finite state space under nonlinear expectations.~We define so-called $Q$-operators as an extension of $Q$-matrices or rate matrices to a nonlinear setup, where the nonlinearity is due to model uncertainty. The main result gives a full characterization of convex $Q$-operators in terms of a positive maximum principle, a dual representation by means of $Q$-matrices, continuous-time Markov chains under convex expectations and nonlinear ordinary differential equations. This extends a classical characterization of generators of Markov chains to the case of model uncertainty in the generator. We further derive a primal and dual representation of the convex semigroup arising from a Markov chain under a convex expectation via the Fenchel-Legendre transformation of its generator. We illustrate the results with several numerical examples, where we compute price bounds for European contingent claims under model uncertainty in terms of the rate matrix.\\
\smallskip
\noindent \textit{Key words:} Nonlinear expectation, imprecise Markov chain, model uncertainty, nonlinear ODE, generator of a nonlinear semigroup

\smallskip
\noindent \emph{AMS 2010 Subject Classification:} 60J27; 60J35; 47H20; 34A34

\end{abstract}

\maketitle

\section{Introduction and main result}

The notion of a nonlinear expectation was introduced by Peng \cite{MR2143645}. Since then, nonlinear expectations have been widely used in order to describe model uncertainty in a probabilistic framework. In mathematical finance, model uncertainty or ambiguity typically appears due to incomplete information in the financial market. In contrast to other disciplines like physics, where experiments can be repeated under similar conditions arbitrarily often, financial markets evolve dynamically and usually do not allow for repetition. A typical example for ambiguity is uncertainty with respect to the parameters (drift, volatility, etc.) of the driving process, which appears when, due to statistical estimation methods, only a confidence interval for the parameter is known. From a theoretical point of view, this leads to the task of modelling stochastic processes under incomplete information (imprecision). This topic has been extensively investigated since the beginning of this century. Prominent examples of nonlinear expectations related to stochastic processes, appearing in mathematical finance, are the g-expectation, see Coquet et al.~\cite{MR1906435}, the G-expectation or $G$-Brownian motion introduced by Peng \cite{MR2397805},\cite{MR2474349} and $G$-L\'evy processes, cf. Hu and Peng \cite{PengHu}, Neufeld and Nutz \cite{NutzNeuf} and Denk et al.~\cite{dkn2}. Concepts that are related to nonlinear expectations are coherent monetary risk measures as introduced by Artzner et al.~\cite{MR1850791} and Delbaen \cite{MR2011534},\cite{MR1929369}, convex monetary risk measures introduced by F\"ollmer and Schied \cite{MR1932379} and Frittelli and Rosazza Gianin \cite{frittelli2002putting}, coherent upper previsions introduced by Walley \cite{MR1145491}, convex upper previsions cf.~Pelessoni and Vicig \cite{MR2027088},\cite{MR2139424}. Further concepts, that describe model uncertainty, are Choquet capacites (see e.g. Dellacherie and Meyer \cite{MR521810}), game-theoretic probability by Vovk and Shafer \cite{MR3287403} and niveloids (see e.g. Cerreia-Vioglio et al.~\cite{MR3153589}).\\

In \cite{MR2143645}, Peng introduces a first notion of Markov chains under nonlinear expectations.~However, the existence of stochastic processes under nonlinear expectations has only been considered in terms of finite dimensional nonlinear marginal distributions, whereas completely path-dependent functionals could not be regarded. Markov chains under model uncertainty have been considered amongst others by Hartfiel \cite{MR1725607}, {\v{S}}kulj \cite{MR2556123} and De Cooman et al.~\cite{MR2535022}. Hartfiel \cite{MR1725607} considers so-called Markov set-chains in discrete time, using matrix intervals in order to describe model uncertainty in the transition matrices. Later, {\v{S}}kulj \cite{MR2556123} approached Markov chains under model uncertainty using Choquet capacities, which results in higher-dimensional matrices on the power set, while De Cooman et al. \cite{MR2535022} considered imprecise Markov chains using an operator-theoretic approach with upper and lower expectations. In \cite[Example 5.3]{MR3824739}, Denk et al.~describe model uncertainty in the transition matrix via a nonlinear transition operator, which, together with the results obtained in \cite{MR3824739}, allows the construction of discrete-time Markov chains on the canonical path space. In continuous time, in particular, computational aspects of sublinear imprecise Markov chains, have been studied amongst others by Krak et al.~ \cite{MR3679158} and {\v{S}}kulj \cite{MR3285527}.\\

In this paper, we investigate continuous-time Markov chains with a finite state space under convex expectations. The main tools, we use in our analysis, are convex duality, so-called Nisio semigroups (cf. Nisio \cite{MR3290231}, Denk et al.~\cite{dkn2}, Nendel and R\"ockner \cite{roecknen}) and a convex version of Kolmogorov's extension theorem, see Denk et al.~\cite{MR3824739}. Restricting the time parameter in the present work to the set of natural numbers leads to a discrete-time Markov chain, in the sense of \cite[Example 5.3]{MR3824739}. A concept that is related to Markov chains under nonlinear expectations are backward stochastic differential equations (BSDEs) on Markov chains by Cohen and Elliot \cite{MR2446692},\cite{MR2594366}.\\

The aim of this paper is to extend the classical relation between Markov chains, $Q$-matrices or rate matrices and ordinary differential equations to the case of model uncertainty. This allows to compute price bounds for contingent claims under model uncertainty numerically by solving ordinary differential equations using Runge-Kutta methods or, in the simplest case, the Euler method. In Section \ref{sec:numex} (see Example \ref{ex:numode}), we illustrate how this method can be used to compute price bounds for European contingent claims, where we consider an underlying Markov chain, which is, firstly, a discrete version of a Brownian motion with uncertain drift (cf.~Coquet et al.~\cite{MR1906435}) and, secondly, a discrete version of a Brownian motion with uncertain volatility (cf.~Peng \cite{MR2397805},\cite{MR2474349}). The aforementioned pricing ODE related to a Markov chain under a nonlinear expectation is a spatially discretized version of a Hamilton-Jacobi-Bellman equation, and is related, via a dual representation, to a control problem where, roughly speaking, ``nature'' tries to control the system into the worst possible scenario (see Remark \ref{rem:dualrep}). The dual representation as a control problem gives rise to a second numerical scheme for the computation of prices of European contingent claims under model uncertainty, which is also discussed in Section \ref{sec:numex} (see Example \ref{ex:numdualrep}).\\

Given a measurable space $(\Om,\FF)$, we denote the space of all bounded measurable functions $\Om\to \R$ by $\LL^\infty(\Om,\FF)$. A \textit{nonlinear expectation} is then a functional $\EE\colon \LL^\infty(\Om,\FF)\to\R$, which satisfies $\EE(X)\le \EE(Y)$ whenever $X(\omega)\le Y(\omega)$ for all $\omega\in \Omega$, and $\EE(\alpha 1_\Omega)=\alpha$ for all $\alpha\in\mathbb{R}$. If $\EE$ is additionally convex, we say that $\EE$ is a \textit{convex expectation}. It is well known (see e.g. Denk et al.~\cite{MR3824739} or F\"ollmer and Schied \cite{MR2779313}) that every convex expectation $\EE$ admits a dual representation in terms of finitely additive probability measures. If $\EE$, however, even admits a dual representation in terms of (countably additive) probability measures, we say that $(\Om,\FF,\EE)$ is a \textit{convex expectation space}. More precisely, we say that $(\Om,\FF,\EE)$ is a \textit{convex expectation space} if there exists a set $\PP$ of probability measures on $(\Om,\FF)$ and a family $(\al_\P)_{\P\in \PP}\subset [0,\infty)$ with $\inf_{\P\in \PP}\al_\P=0$ such that
\[
 \EE(X)=\sup_{\P\in \PP} \big(\E_\P(X)-\al_\P\big)
\]
for all $X\in \LL^\infty(\Om,\FF)$. Here, $\E_\P$ denotes the expectation w.r.t. a probability measure $\P$ on $(\Om,\FF)$. If $\al_\P=0$ for all $\P\in \PP$, we say that $(\Om,\FF,\EE)$ is a \textit{sublinear expectation space}. Here, the set $\PP$ represents the set of all models that are relevant under the expectation $\EE$. In the case of a sublinear expectation space, the functional $\EE$ is the best case among all plausible models $\PP$. In the case of a convex expectation space, the functional $\EE$ is a weighted best case among all plausible models $\PP$ with an additional penalization term $\al_\P$ for every $\P\in \PP$. Intuitively, $\al_\P$ can be seen as a measure for how much importance we give to the prior $\P\in \PP$ under the expectation $\EE$. For example, a low penalization, i.e.~$\al_\P$ close or equal to $0$, gives more importance to the model $\P\in \PP$ than a high penalization.\\

If a nonlinear expectation $\EE$ is sublinear, then $\rho(X):=\EE(-X)$ defines a coherent monetary risk measure as introduced by Artzner et al.~\cite{MR1850791} and Delbaen \cite{MR2011534},\cite{MR1929369}, see also F\"ollmer and Schied \cite{MR2779313} for an overview of monetary risk measures. Moreover, if $\EE$ is a sublinear expectation, then $\EE$ is a coherent upper prevision, cf. Walley \cite{MR1145491}, and vice versa. We would further like to mention that there is also a one-to-one relation between the following three concepts: convex expectations, convex upper previsions, cf.~Pelessoni and Vicig \cite{MR2027088},\cite{MR2139424}, and convex risk measures, cf.~F\"ollmer and Schied \cite{MR1932379} and Frittelli and Rosazza Gianin \cite{frittelli2002putting}.\\

Throughout, we consider a finite non-empty state space $S$ with cardinality $d:=|S|\in \N$. We endow $S$ with the discrete topology $2^S$ and w.l.o.g. assume that $S=\{1,\ldots, d\}$. The space of all bounded measurable functions $S\to \R$ can therefore be identified by $\R^d$ via
\[
 u=(u_1,\ldots, u_d)^T \quad \text{with }u_i:=u(i)\quad \text{for all }i\in \{1,\ldots, d\}
\]
Therefore, a bounded measurable function $u$ will always be denoted as a vector of the form $u=(u_1,\ldots, u_d)^T\in \R^d$, where $u_i$ represents the value of $u$ in the state  $i\in \{1,\ldots,d\}$. On $\R^d$, we consider the norm
\[
 \|u\|_\infty:=\max_{i=1,\ldots, d}|u_i|
\]
for a vector $u\in \R^d$. Moreover, for $\al\in \R$, the vector $\al\in \R^d$ denotes the constant vector $u\in \R^d$ with $u_i=\al$ for all $i\in \{1,\ldots, d\}$. For a matrix $a=(a_{ij})_{1\leq i,j\leq d}\in \R^{d\times d}$, we denote by $\|a\|$ the operator norm of $a\colon \R^d\to \R^d$ w.r.t. the norm $\|\cdot \|_\infty$, i.e.
\[
 \|a\|=\sup_{v\in \R^d\setminus\{0\}} \frac{\|av\|_\infty}{\|v\|_\infty}=\max_{i=1,\ldots, d}\bigg(\sum_{j=1}^d |a_{ij}|\bigg).
\]
Inequalities of vectors are always understood componentwise, i.e. for $u,v\in \R^d$
\[
 u\leq v \iff \forall i\in \{1,\ldots, d\}\colon u_i\leq v_i.
\]
All concepts in $\R^d$ that include inequalities are to be understood w.r.t. the latter partial ordering. For example, a vector field $F\colon \R^d\to \R^d$ is called convex if
\[
 F_i\big(\la u+(1-\la)v\big)\leq \la F_i(u)+(1-\la)F_i(v)
\]
for all $i\in \{1,\ldots, d\}$, $u,v\in \R^d$ and $\la\in [0,1]$. A vector field $F$ is called sublinear if it is convex and positive homogeneous (of degree $1$). Moreover, for a set $M\subset \R^d$ of vectors, we write $u=\sup M$ if $u_i=\sup_{v\in M} v_i$ for all $i\in \{1,\ldots, d\}$.\\


A matrix $q=(q_{ij})_{1\leq i,j\leq d}\in \R^{d\times d}$ is called a \textit{$Q$-matrix} or \textit{rate matrix} if it satisfies the following conditions:
\begin{enumerate}
	\item[(i)] $q_{ii}\leq 0$ for all $i\in\{1,\ldots, d\}$,
	\item[(ii)] $q_{ij}\geq 0$ for all $i,j\in \{1,\ldots, d\}$ with $i\neq j$,
	\item[(iii)] $\sum_{j=1}^dq_{ij}=0$ for all $i\in \{1,\ldots, d\}$.
\end{enumerate}
It is well known that every continuous-time Markov chain with certain regularity properties at time $t=0$ can be related to a $Q$-matrix and vice versa. More precisely, for a matrix $q\in \R^{d\times d}$ the following statements are equivalent:
\begin{enumerate}
 \item[$(i)$] $q$ is a $Q$-matrix.
 \item[$(ii)$] There is a Markov chain $\big(\Om,\FF,(\P_1,\ldots, \P_d),(X_t)_{t\geq 0}\big)$ such that
  \[
   q u_0=\lim_{h\searrow 0} \frac{\E\big(u_0(X_h)\big)-u_0}{h}\quad \text{for all }u_0\in \R^d,
  \]
   where $u_0(i)$ is the $i$-th component of $u_0$ for $i\in \{1,\ldots, d\}$, $\P_i$ stands for the probability measure under which the Markov chain $(X_t)_{t\geq 0}$ satisfies $\P_i(X_0=i)=1$, for $i\in \{1,\ldots, d\}$, and $$\E(Y):=\big(\E_{\P_1}(Y),\ldots \E_{\P_d}(Y)\big)^T\in \R^d$$ for any bounded random variable $Y\colon \Om\to \R$. 
\end{enumerate}
In this case, for each vector $u_0\in \R^d$, the function $u\colon [0,\infty)\to \R^d,\; t\mapsto \E\big(u_0(X_t)\big)$ is the unique classical solution $u\in C^1\big([0,\infty);\R^d\big)$ to the initial value problem
 \begin{align*}
   u'(t)&=q u(t),\quad t\geq 0,\\
   u(0)&=u_0,
  \end{align*}
i.e. $u(t)=e^{tq}u_0$ for all $t\geq 0$, where $e^{tq}$ is the matrix exponential of $tq$. We refer to Norris \cite{MR1600720} for a detailed illustration of this relation.\\

We say that a (possibly nonlinear) operator $\QQ\colon \R^d\to \R^d$ satisfies the \textit{positive maximum principle} if, for every $u=(u_1,\ldots, u_d)^T\in \R^d$ and $i\in \{1,\ldots, d\}$,
$$(\QQ u)_i\leq 0\quad \text{whenever} \quad u_i\geq u_j\text{ for all }j\in \{1,\ldots,d\}.$$
This notion is motivated by the positive maximum priciple for generators of Feller processes, see e.g.~Jacob \cite[Equation (0.8)]{MR1873235}. Notice that a matrix $q\in \R^{d\times d}$ is a $Q$-matrix if and only if it satisfies the positive maximum principle and $q 1=0$, where $1:=(1,\ldots, 1)^T\in \R^d$ denotes the constant $1$ vector. Indeed, property (iii) in the definition of a $Q$-matrix is just a reformulation of $q1=0$. Moreover, if $q$ satisfies the positive maximum principle, then $q_{ii}=(qe_i)_i\leq 0$ for all $i\in \{1,\ldots, d\}$ and $-q_{ij}=(q(-e_i))_j\leq 0$ for all $i,j\in \{1,\ldots, d\}$ with $i\neq j$. On the other hand, if $q$ is a $Q$-matrix, $u=(u_1,\ldots, u_d)^T\in \R^d$ and $i\in \{1,\ldots, d\}$ with $u_i\geq u_j$ for all $j\in \{1,\ldots,d\}$, then, $(qu)_i=\sum_{j=1}^dq_{ij}u_j\leq u_i\sum_{j=1}^d q_{ij}=0$, which shows that $q$ satisfies the positive maximum principle.\\

In order to state the main result, we need the following definitions.

\begin{definition}\label{defqop}
A (possibly nonlinear) map $\QQ\colon \R^d\to \R^d$ is called a \textit{$Q$-operator} if the following conditions are satisfied:
\begin{enumerate}
 \item[$(i)$] $(\QQ \la e_i)_i\leq 0$ for all $\la>0$ and all $i\in \{1,\ldots, d\}$,
 \item[$(ii)$] $\big(\QQ(-\la e_j)\big)_i\leq 0$ for all $\la>0$ and all $i,j\in \{1,\ldots, d\}$ with $i\neq j$,
 \item[$(iii)$] $\QQ\al=0$ for all $\al\in \R$, where we identify $\al$ with $(\al,\ldots, \al)^T\in \R^d$.
\end{enumerate}
\end{definition}

\begin{definition}\label{markovprocess}
 A \textit{convex Markov chain} is a quadruple $\big(\Om,\FF,\EE,(X_t)_{t\geq 0}\big)$ that satisfies the following conditions:
 \begin{enumerate}
  \item[$(i)$] $(\Om,\FF)$ is a measurable space.
  \item[$(ii)$] $X_t\colon \Om\to \{1,\ldots, d\}$ is $\FF$-measurable for all $t\geq 0$.
  \item[$(iii)$] $\EE=(\EE_1,\ldots, \EE_d)^T$, where $(\Om,\FF,\EE_i)$ is a convex expectation space for all $i\in \{1,\ldots, d\}$ and $\EE\big(u_0(X_0)\big)=u_{0}$. Here and in the following we make use of the notation
  \[
   \EE(Y):=\big(\EE_1(Y),\ldots, \EE_d(Y)\big)^ T\in \R^d
  \]
  for $Y\in \LL^ \infty(\Om,\FF)$. 
  \item[$(iv)$] The following version of the Markov property is satisfied: For all $s,t\geq 0$, $n\in \N$, $0\leq t_1<\ldots <t_n\leq s$ and $v_0\in \big(\R^d\big)^{(n+1)}$,
  \[
   \EE\big(v_0(Y,X_{s+t})\big)=\EE\left[\EE_{X_s,t}\big(v_0(Y,\, \cdot\,)\big)\right],
  \]
  where $Y:=(X_{t_1},\ldots,X_{t_n})$ and $\EE_{i,t} (u_0):=\EE_i\big(u_0(X_t)\big)$ for all $u_0\in \R^d$ and $i\in\{1,\ldots, d\}$.
 \end{enumerate}
 We say that the Markov chain $\big(\Om,\FF,\EE,(X_t)_{t\geq 0}\big)$ is \textit{linear} or \textit{sublinear} if the mapping $\EE\colon \LL^\infty(\Om,\FF)\to \R^d$ is additionally linear or sublinear, respectively.
\end{definition}

The Markov property given in $(iv)$ of the previous definition is the nonlinear analogon of the classical Markov property without using conditional expectations. Notice that, due to the nonlinearity of the expectation, the definition and, in particular, the existence of a conditional (nonlinear) expectation is quite involved, which is why we avoid to introduce this concept.\\

In line with \cite[Definition 5.1]{MR3824739}, we say that a (possibly nonlinear) map $\EE\colon \R^d\to \R^d$ is a \textit{kernel}, if $\EE$ is \textit{monotone}, i.e. $\EE(u)\leq \EE(v)$ for all $u,v\in \R^d$ with $u\leq v$, and $\EE$ \textit{preserves constants}, i.e. $\EE(\al)=\al$ for all $\al\in \R$.

\begin{definition}\label{semigroup}
  A family $\SS=\big(\SS(t)\big)_{t\geq 0}$ of (possibly nonlinear) operators $\SS(t)\colon \R^d\to \R^d$ is called a \textit{semigroup} if
 \begin{enumerate}
  \item[(i)] $\SS(0)=I$, where $I=I_d$ is the $d$-dimensional identity matrix,
  \item[(ii)] $\SS(s+t)=\SS(s)\SS(t)$ for all $s,t\geq 0$.
  \end{enumerate}
  Here and throughout, we make use of the notation $\SS(s)\SS(t):=\SS(s)\circ\SS(t)$. If, additionally, $\SS(h)\to I$ uniformly on compact sets as $h\searrow 0$, we say that the semigroup $\SS$ is \textit{uniformly continuous}.  We call $\SS$ \textit{Markovian} if $\SS(t)$ is a \textit{kernel} for all $t\geq 0$. We say that $\SS$ is \textit{linear}, \textit{sublinear} or \textit{convex} if $\SS(t)$ is linear, sublinear or convex for all $t\geq 0$, respectively.
\end{definition}

\begin{definition}\label{nisiosg}
 Let $\PP\subset \R^{d\times d}$ be a set of $Q$-matrices and $f=(f_q)_{q\in \PP}$ a family of vectors with $\sup_{q\in \PP} f_q=f_{q_0}=0$ for some $q_0\in \PP$. We denote by
 \[
  S_q(t)u_0:=e^{qt}u_0+\int_0^t e^{qs}f_q\, {\rm d}s=u_0+\int_0^t e^{sq}\big(qu_0+f_q\big)\, {\rm d}s
 \]
 for $t\geq 0$, $u_0\in \R^d$ and $q\in \PP$. Then, $S_q=\big(S_q(t)\big)_{t\geq 0}$ is an affine linear semigroup. We call a semigroup $\SS$ the \textit{(upper) semigroup envelope} (later also \textit{Nisio semigroup}) of $(\PP,f)$ if
  \begin{enumerate}
  \item[(i)] $\SS(t)u_0\geq S_q(t)u_0$ for all $t\geq 0$, $u_0\in \R^d$ and $q\in \PP$,
  \item[(ii)] for any other semigroup $\mathscr T$ satisfying (i) we have that $\SS(t)u_0\leq \mathscr T(t)u_0$ for all $t\geq 0$ and $u_0\in \R^d$.
  \end{enumerate}
  That is, the semigroup envelope $\SS$ is the smallest semigroup that dominates all semigroups $(S_q)_{q\in \PP}$.
\end{definition}

The following main theorem gives a full characterization of convex $Q$-operators.

\begin{theorem}\label{main}
 Let $\QQ\colon \R^d\to \R^d$ be a mapping. Then, the following statements are equivalent:
 \begin{enumerate}
  \item[$(i)$] $\QQ$ is a convex $Q$-operator.
  \item[$(ii)$] $\QQ$ is convex, satisfies the positive maximum principle and $\QQ \al=0$ for all $\al\in \R$, where $\al:=(\al,\ldots, \al)^T\in \R^d$.
  \item[$(iii)$] There exists a set $\PP\subset \R^{d\times d}$ of $Q$-matrices and a family $f=(f_q)_{q\in \PP}\subset \R^d$ of vectors with $\sup_{q\in \PP} f_q=f_{q_0}=0$ for some $q_0\in \PP$ such that
  \begin{equation}\label{1}
   \QQ u_0=\sup_{q\in \PP} \big(qu_0+f_q\big)
  \end{equation}
  for all $u_0\in \R^d$, where the suprema are to be understood componentwise.
  \item[$(iv)$] There exists a uniformly continuous convex Markovian semigroup $\SS$ with
    \[
   \QQ u_0=\lim_{h\searrow 0} \frac{\SS(h)u_0-u_0}{h}
  \]
  for all $u_0\in \R^d$.
  \item[$(v)$] There is a convex Markov chain $\big(\Om,\FF,\EE,(X_t)_{t\geq 0}\big)$ such that
  \[
   \QQ u_0=\lim_{h\searrow 0} \frac{\EE\big(u_0(X_h)\big)-u_0}{h}
  \]
  for all $u_0\in \R^d$.
 \end{enumerate}
 In this case, for each initial value $u_0\in \R^d$, the function $u\colon [0,\infty)\to \R^d,\; t\mapsto \EE\big(u_0(X_t)\big)$
 is the unique classical solution $u\in C^1\big([0,\infty);\R^d\big)$ to the initial value problem
 \begin{align}
   u'(t)&=\QQ u(t)=\sup_{q\in \PP} \big(qu(t)+f_q\big),\quad t\geq 0,\label{odemain}\\
   u(0)&=u_0.\notag
  \end{align}
  Moreover, $u(t)=\SS (t)u_0$ for all $t\geq 0$, where $\SS$ is the Markovian semigroup from (iv), and $\SS$ is the semigroup envelope of $(\PP,f)$.
\end{theorem}

\begin{remark}
Consider the situation of Theorem \ref{main}.
\begin{enumerate}
 \item[a)] The dual representation in $(iii)$ gives a model uncertainty interpretation to $Q$-operators. The set $\PP$ can be seen as the set of all plausible rate matrices, when considering the $Q$-operator $\QQ$. For every $q\in \PP$, the vector $f_q\leq 0$ can be interpreted as a penalization, which measures how much importance we give to each rate matrix $q$. The requirement that there exists some $q_0\in \PP$ with $f_{q_0}=0$ can be interpreted in the following way: There has to exist at least one rate matrix $q_0$ within the set of all plausible rate matrices $\PP$ to which we assign the maximal importance, that is the minimal penalization. 
 \item[b)] The semigroup envelope $\SS$ of $(\PP,f)$ can be constructed more explicitly, in particular, an explicit (in terms of $(\PP,f)$) dual representation can be derived. For details, we refer to Section \ref{nisiosec} (Definition \ref{nisiosg1} and Remark \ref{rem:dualrep}). Moreover, we would like to highlight that the semigroup envelope $\SS$ can be constructed w.r.t.~any dual representation $(\PP,f)$ as in $(iii)$ and results in the unique classical solution to \eqref{odemain} independent of the choice of the dual representation $(\PP,f)$ of $\QQ$. This gives, in some cases, the opportunity to efficiently compute the semigroup envelope numerically via its primal/dual representation (see Example \ref{ex:numdualrep}). 
 \item[c)] The same equivalence as in Theorem \ref{main} holds if convexity is replaced by sublinearity in $(i)$, $(ii)$, $(iv)$ and $(v)$ and $f_q=0$ for all $q\in \PP$ in $(iii)$. In this case, the set $\PP$ in $(iii)$ can be chosen to be compact as we will see in the proof of Theorem \ref{main}.
 \item[d)] Theorem \ref{main} extends and includes the well-known relation between (linear) Markov chains, $Q$-matrices and ordinary differential equations.
  \item[e)] A remarkable consequence of Theorem \ref{main} is that every convex Markovian semigroup, which is differentiable at time $t=0$, is the semigroup envelope with respect to the Fenchel-Legendre transformation (or any other dual representation as in $(iii)$ of its generator, which is a convex $Q$-operator.
  \item[f)] Although $\QQ$ has an unbounded convex conjugate, the convex initial value problem
  \begin{equation}\label{eq2}
   u'(t)=\QQ u(t)\; \;\text{for all}\; t\geq 0,\quad u(0)=u_0.
  \end{equation}
  has a unique global solution.
  
  \item[g)] Solutions to \eqref{eq2} remain bounded. Therefore, a Picard iteration or Runge-Kutta methods, such as the explicit Euler method, can be used for numerical computations, and the convergence rate (depending on the size of the initial value $u_0$) can be derived from the a priori estimate from Banach's fixed point theorem.
  
  \item[h)] As in the linear case, by solving the differential equation \eqref{eq2}, one can (numerically) compute expressions of the form
  \[
   u(t)=\EE(u_0(X_t)).
  \]
  under model uncertainty. We illustrate this computation in Example \ref{ex:numode}. 
\end{enumerate} 
\end{remark}

{\bf Structure of the paper.}
In Section \ref{sec2}, we give a proof of the implications $(iv)\Rightarrow (ii)\Rightarrow (i)\Rightarrow (iii)$ of Theorem \ref{main}. The main tool, we use in this part, is convex duality in $\R^d$. In Section \ref{nisiosec}, we prove the implication $(iii)\Rightarrow (iv)$. Here, we use a combination of Nisio semigroups as introduced in \cite{MR3290231}, a Kolmogorov-type extension theorem for convex expectations derived in \cite{MR3824739} and the theory of ordinary differential equations. In Section \ref{sec:numex}, we use two different numerical methods, based on the results from Section \ref{nisiosec}, in order to compute price bounds for European contingent claims, where the underlying is a discrete version of a Brownian motion with drift undertainty ($g$-framework) and volatility uncertainty ($G$-framework).


\section{Proof of $(iv)\Rightarrow (ii)\Rightarrow (i)\Rightarrow (iii)$}\label{sec2}

We say that a set $\PP\subset \R^{d\times d}$ of matrices is \textit{row-convex} if, for any diagonal matrix $\la\in \R^{d\times d}$ with $\la_i:=\la_{ii}\in [0,1]$ for all $i\in \{1,\ldots, d\}$,
\[
 \la p+(I-\la)q\in \PP\quad \text{for all }p,q\in \PP,
\]
where $I=I_d\in \R^{d\times d}$ is the $d$-dimensional identity matrix. Notice that, for all $i\in \{1,\ldots, d\}$, the $i$-th row of the matrix $\la p+(I-\la)q$ is the convex combination of the $i$-th row of $p$ and $q$ with $\la_i$. Therefore, a set $\PP$ is row-convex if for all $p,q\in \PP$ the convex combination with different $\la\in [0,1]$ in every row is again an element of $\PP$. For example, the set of all $Q$-matrices is row-convex.

\begin{remark}\label{conjfunct}
Let $\QQ$ be a convex $Q$-operator. For every matrix $q\in \R^{d\times d}$, let
  \[
   \QQ^*(q):=\sup_{u\in \R^d} \big(qu -\QQ(u)\big)\in [0,\infty]^d
  \]
   be the \textit{conjugate function} of $\QQ$. Notice that $0\leq \QQ^*(q)$ for all $q\in \R^{d\times d}$, since $\QQ(0)=0$. Let
   \[
    \PP:=\{q\in \R^{d\times d}\, |\, \QQ^*(q)<\infty\}
   \]
   and $f_q:=-\QQ^*(q)$ for all $q\in \PP$. Then, the following facts are well-known results from convex duality theory in $\R^d$.
\begin{enumerate}
\item[a)] The set $\PP$ is row-convex and the mapping $\PP\to \R^d,\; q\mapsto f_q$ is continuous.
  \item[b)] Let $M>0$ and $\PP_M:=\{q\in \R^{d\times d}\, |\, \QQ^*(q)\leq M\}$. Then, $\PP_M\subset \R^{d\times d}$ is compact and row-convex. Therefore,

 \begin{equation}\label{qmaxop}
  \QQ_M\colon \R^d\to \R^d, \quad u\mapsto \max_{q\in \PP_M}\big( qu+f_q\big)
 \end{equation}
  defines a convex operator, which is Lipschitz continuous. Notice that the maximum in \eqref{qmaxop} is to be understood componentwise. However, for fixed $u_0\in \R^d$, the maximum can be attained, simultaneously in every component, by a single element of $\PP_M$, since $\PP_M$ is row-convex, i.e.,~for all $u_0\in \R^d$, there exists some $q_0\in \PP_M$ with
  \[
   \QQ_M=q_0 u_0+f_{q_0}.
  \]
  \item[c)] Let $R>0$. Then, there exists some $M>0$, such that
 \[
  \QQ u_0=\max_{q\in \PP_M} \big(qu_0+f_q\big)=\QQ_M u_0
 \]
 for all $u_0\in \R^d$ with $\|u_0\|_\infty\leq R$. In particular, $\QQ$ is locally Lipschitz continuous and admits a representation of the form
 \[
  \QQ u_0=\max_{q\in \PP} \big(qu_0+f_q\big)
 \]
 for all $u_0\in \R^d$, where, for fixed $u_0\in \R^d$, the maximum can be attained,  simultaneously in every component, by a single element of $\PP$. In particular, there exists some $q_0\in \PP$ with $f_{q_0}=\sup_{q\in \PP}f_q=\QQ(0)=0$.
 \end{enumerate}
\end{remark}

\begin{proof}[Proof of Theorem \ref{main}]\ \\
 $(iv)\Rightarrow (ii)$: As $\EE_i$ is a convex expectation for all $i\in \{1,\ldots, d\}$, it follows that the operator $\QQ$ is convex with $\QQ \al=0$ for all $\al\in \R$. Now, let $u_0\in \R^d$ and $i\in \{1,\ldots, d\}$ with $u_{0,i}\geq u_{0,j}$ for all $j\in \{1,\ldots, d\}$. Let $\al >0$ be such that
 \[
  \|u_0+\al\|_\infty=\big(u_0+\al\big)_i=u_{0,i}+\al
 \]
 and $v_0:=u_0+\al$. Then,
 \[
  \QQ v_0=\lim_{h\searrow 0}\frac{\EE\big(u_0(X_h)+\al\big)-v_0}{h}=\lim_{h\searrow 0}\frac{\EE\big(u_0(X_h)\big)-u_0}{h}=\QQ u_0.
 \]
 Assume that $\big(\QQ u_0\big)_i>0$. Then, there exists some $h>0$ such that
 \[
  \EE_i\big(v_0(X_h)\big)-v_{0,i}>0,
 \]
 i.e.
 \[
  \big\|\EE\big(v_0(X_h)\big)\big\|_{\infty}\geq \EE_i\big(v_0(X_h)\big)>v_{0,i}=\|v_0\|_\infty,
 \]
 which is a contradiction to
 \[
  \big\|\EE\big(v_0(X_h)\big)\big\|_{\infty}\leq \|v_0\|_\infty.
 \]
 This shows that $\QQ$ satisfies the positive maximum principle.\\
$(ii) \Rightarrow (i)$: This follows directly from the positive maximum principle, considering the vectors $\la e_i$ and $-\la e_i$ for all $\la>0$ and $i\in \{1,\ldots, d\}$.\\
  $(i)\Rightarrow (iii)$: Let $\QQ$ be a convex $Q$-operator. Moreover, let $\PP$ and $f=(f_q)_{q\in \PP}$ be as in Remark \ref{qmaxop}. 
 Then, by Remark \ref{qmaxop} c), it only remains to show that every $q\in \PP$ is a $Q$-matrix. To this end, fix an arbitrary $q\in \PP$. Then, for all $\al\in \R$,
 \[
  q\al=\frac{1}{\la} q(\la\al)\leq \frac{1}{\la}\big(\QQ(\la\al)+\QQ^*(q)\big)=\frac{1}{\la}\QQ^*(q)\to 0\quad \text{as }\la\to \infty.
 \]
 Therefore, $q\al\leq 0$ for all $\al\in \R$. Since, $q$ is linear, it follows $q1=0$. Now, let $i\in \{1,\ldots, d\}$. Then, by definition of a $Q$-operator, we obtain that
\[
 q_{ii}\leq \frac{1}{\la} \big(\QQ(\la e_i)+\QQ^*(q)\big)_i\leq \frac{1}{\la}\big(\QQ^*(q)\big)_i\to 0 \quad \text{as }\la\to \infty,
\]
 i.e. $q_{ii}\leq 0$. Now, let $i,j\in \{1,\ldots, d\}$ with $i\neq j$. Then, again by definition of a $Q$-operator, it follows that 
\[
 -q_{ij}\leq \frac{1}{\la} \big(\QQ(-\la e_i)+\QQ^*(q)\big)_j\leq \frac{1}{\la}\big(\QQ^*(q)\big)_j\to 0\quad \text{as }\la\to \infty,
\]
 i.e. $q_{ij}\geq 0$. Therefore, $q$ is a $Q$-matrix.\\
It remains to show $(iii)\Rightarrow(iv)$, which is done in the entire next section.
\end{proof}


 
 \section{Proof of $(iii)\Rightarrow (iv)$}\label{nisiosec}

Throughout, let $\PP\subset \R^{d\times d}$ be a set of $Q$-matrices and $f=(f_q)_{q\in \PP}\subset \R^d$ with $\sup_{q\in \PP}f_q=f_{q_0}=0$, for some $q_0\in \PP$, such that
\[
 \QQ\colon \R^d\to \R^d,\quad u\mapsto \sup_{q\in \PP} \big(q u+f_q\big)
\]
is well-defined.
For every $q\in \PP$, we consider the linear ODE
\begin{equation}\label{linode}
 u'(t)= q u(t)+f_q,\quad \text{for }t\geq 0,
\end{equation}
 with $u(0)=u_0\in \R^d$. Then, by a variation of constant, the solution to \eqref{linode} is given by 
 \begin{equation}\label{duhamel}
  u(t)=e^{qt}u_0+\int_0^t e^{qs}f_q\, {\rm d}s=u_0+\int_0^t e^{sq}\big(qu_0+f_q\big)\, {\rm d}s=:S_q(t)u_0
 \end{equation}
 for $t\geq 0$, where $e^{tq}\in \R^{d\times d}$ is the matrix exponential of $tq$ for all $t\geq 0$. Then, the family $S_q=\big(S_q(t)\big)_{t\geq 0}$ defines a uniformly continuous semigroup of affine linear operators (see Definition \ref{semigroup}).
  
  \begin{remark}
  Note that, for all $q\in \PP$ and $t\geq 0$, the matrix exponential $e^{tq}\in \R^{d\times d}$ is a \textit{stochastic matrix}, i.e.
\begin{enumerate}
	\item[(i)] $\big(e^{tq}\big)_{ij}\geq 0$ for all $i,j\in \{1,\ldots, d\}$,
	\item[(ii)] $e^{tq}1=1$.
\end{enumerate}
Therefore, $e^{tq}\in \R^{d\times d}$ is a linear kernel, i.e. $e^{tq}u_0\leq e^{tq}v_0$ for all $u_0,v_0\in \R^d$ with $u_0\leq v_0$ and $e^{tq}\al=\al$ for all $\al\in \R$, which implies that $S_q(t)$ is monotone for all $q\in \PP$ and $t\geq 0$.
  \end{remark}
  
For the family $(S_q)_{q\in \PP}$ or, more precisely, for $(\PP,f)$, we will now construct the \textit{Nisio semigroup}, and show that it gives rise to the unique classical solution to the nonlinear ODE \eqref{odemain}. To this end, we consider the set of finite partitions
\[
P:=\big\{\pi\subset [0,\infty)\, \big|\, 0\in \pi, |\pi|<\infty\big\}. 
\]
The set of partitions with end-point $t\geq 0$ will be denoted by $P_t$, i.e. $P_t := \{\pi \in P\, |\, \max \pi = t\}$. Notice that
\[
 P=\bigcup_{t\geq 0} P_t.
\]
For all $h\geq 0$ and $u_0\in \R^d$, we define
\[
 \EE_h u_0:=\sup_{q\in \PP} S_q(h) u_0,
\]
where the supremum is taken componentwise. Note that $\EE_h$ is well-defined since
\[
 S_q(h)u_0=e^{hq}u_0+\int_0^h e^{sq}f_q\, {\rm d}s\leq e^{hq}u_0\leq \|u_0\|_\infty
\]
for all $q\in \PP$, $h\geq 0$ and $u_0\in \R^d$, where we used the fact that $e^{hq}$ is a kernel. Moreover, $\EE_h$ is a convex kernel, for all $h\geq 0$, as it is monotone and
\[
 \EE_h \al=\al+\sup_{q\in \PP}\int_0^h e^{sq}f_q\, {\rm d}s=\al
\]
for all $\al\in \R$, where we used the fact that there is some $q_0\in \PP$ with $f_{q_0}=0$. For a partition $\pi=\{t_0,t_1,\ldots,t_m\}\in P$ with $m\in \N$ and $0=t_0<t_1<\ldots<t_m$, we set
\[
 \EE_\pi:=\EE_{t_1-t_0}\ldots \EE_{t_m-t_{m-1}}.
\]
Moreover, we set $\EE_{\{0\}}:=\EE_0$. Then, $\EE_\pi$ is a convex kernel for all $\pi\in P$ since it is a concatenation of convex kernels.

\begin{definition}\label{nisiosg1}
 The \textit{Nisio semigroup} $\SS=\big(\SS(t)\big)_{t\geq 0}$ of $(\PP,f)$ is defined by
\[
 \SS(t)u_0:=\sup_{\pi\in P_t} \EE_\pi u_0
\]
for all $u_0\in \R^d$ and $t\geq 0$. 
\end{definition}

Notice that $\SS(t)\colon \R^d\to \R^d$ is well-defined and a convex kernel for all $t\geq 0$ since $\EE_\pi$ is a convex kernel for all $\pi\in P$. In many of the subsequent proofs, we will first concentrate on the case, where the family $f$ is bounded and then use an approximation of the Nisio semigroup by means of other Nisio semigroups. This approximation procedure is specified in the following remark.

\begin{remark}\label{uniclass1}
 Let $M\geq 0$, $\PP_M:=\big\{q\in \PP\, \big|\, \|f_q\|_\infty\leq M\big\}$ and $f_M:=(f_q)_{q\in \PP_M}$. Notice that, by assumption, there exists some $q_0\in \PP$ with $f_{q_0}=0$, which implies that $q_0\in \PP_M$. Since $\PP_M\subset \PP$ (and by definition of $f_M$), the operator
 \[
  \QQ_M\colon \R^d\to \R^d,\quad v\mapsto \sup_{q\in \PP_M} \big(qv +f_q\big)
 \]
 is well-defined. Let $\SS_M$ be the Nisio semigroup w.r.t. $(\PP_M,f_M)$ for all $M\geq 0$. Since
\[
 \bigcup_{M\geq 0}\PP_M=\PP,
\]
it follows that $\QQ_M\nearrow \QQ$ and $\SS_M(t)\nearrow \SS(t)$, for all $t\geq 0$, as $M\to \infty$.
Moreover, for all $q\in \PP_M$ and $u_0\in \R^d$ with $\|u_0\|_\infty=1$,
 \[
 q u_0\leq \QQ u_0-f_q\leq \|\QQ u_0\|_\infty+\|f_q\|_\infty\leq M+\max_{v\in \mathbb{S}^{d-1}}\|\QQ v\|_\infty,
 \]
 where $\mathbb{S}^{d-1}:=\{v\in \R^d\, |\, \|v\|_\infty=1\}$ and, in the last step, we used the fact that $\QQ\colon \R^d\to \R^d$ is convex and therefore continuous. This implies that the set $\PP_M$ is bounded in the sense that $\sup_{q\in \PP_M}\|q\|<\infty$. In particular,
 \begin{equation}\label{bla1}
  \sup_{q\in \PP_M} \|q u_0+f_q\|_\infty\leq \sup_{q\in \PP_M}\big(\|q\|\|u_0\|_\infty+\|f_q\|_\infty\big)\leq M+ \sup_{q\in \PP_M}\|q\|\|u_0\|_\infty<\infty
 \end{equation}
  for all $u_0\in \R^d$.
\end{remark}

\begin{lemma}\label{Ehcont}
 Assume that the family $f$ is bounded, i.e. $(\PP,f)=(\PP_M,f_M)$ for some $M\geq 0$. Then, for all $u_0\in \R^d$, the mapping $[0,\infty)\to \R^d,\; h\mapsto \EE_h u_0$ is Lipschitz continuous.
\end{lemma}

\begin{proof}
 Let $u_0\in \R^d$ and $0\leq h_1<h_2$. Then, by \eqref{duhamel}, for all $q\in \PP$ we have that
 \begin{align*}
  \|S_q(h_2)u_0 -S_q(h_1)u_0\|_\infty\leq \int_{h_1}^{h_2}\big\|e^{qs}(qu_0+f_q)\big\|_\infty\, {\rm d}s\leq (h_2-h_1)\|qu_0+f_q\|_\infty,
 \end{align*}
 which implies that
 \begin{equation}\label{bla2}
  \|\EE_{h_2}u_0-\EE_{h_1}u_0\|_\infty\leq \sup_{q\in \PP} \|S_q(h_2)u_0 -S_q(h_1)u_0\|_\infty \leq (h_2-h_1)\bigg(\sup_{q\in \PP}\|qu_0+f_q\|_\infty\bigg).
 \end{equation}
 Note that $\sup_{q\in \PP}\|qu_0+f_q\|_\infty<\infty$ by \eqref{bla1}.
\end{proof}

\begin{lemma}\label{unicont}
Assume that the family $f$ is bounded. Then,
 \[
  \|\SS(t)u_0-u_0\|_\infty\leq t\bigg(\sup_{q\in \PP}\|q u_0+f_q\|_\infty\bigg)
 \]
for all $t\geq 0$ and $u_0\in \R^d$. In particular, the map $[0,\infty)\to \R^d,\; t\mapsto \SS(t)u_0$ is Lipschitz continuous for all $u_0\in \R^d$
\end{lemma}

\begin{proof}
 Let $u_0\in \R^d$. 
Then, for any partition $\pi\in P$ of the form $\pi=\{t_0,t_1,\dots,t_m\}$ with $m\in \N$ and $0=t_0<t_1<\ldots <t_m$, \eqref{bla2} together with the fact that $\EE_h$ is a kernel, for all $h\geq 0$, implies that
\begin{align*}
 \|\EE_\pi u_0-u_0\|_\infty&\leq \sum_{k=1}^m \|\EE_{h_k}u_0-u_0\|_\infty\leq\sum_{k=1}^mh_k\bigg(\sup_{q\in \PP}\|q u_0+f_q\|_\infty\bigg)\\
 &= t_m\bigg(\sup_{q\in \PP}\|q u_0+f_q\|_\infty\bigg),
\end{align*}
where $h_k:=t_k-t_{k-1}$ for all $k\in \{1,\ldots, m\}$. By definition of $\SS(t)$, for $t\geq 0$, it follows that
\[
 \|\SS(t)u_0-u_0\|_\infty\leq \sup_{\pi\in P_t} \|\EE_\pi u_0-u_0\|_\infty\leq t\bigg(\sup_{q\in \PP}\|q u_0+f_q\|_\infty\bigg).
\]
Now, let $s,t\geq 0$. Then, since $\SS(h)$ is a kernel for all $h\geq0$, it follows that
\[
 \|\SS(t)u_0-\SS(s)u_0\|_\infty\leq \|\SS(|t-s|)u_0-u_0\|_\infty\leq |t-s|\bigg(\sup_{q\in \PP}\|q u_0+f_q\|_\infty\bigg).
\]

\end{proof}

For a partition $\pi=\{t_0,t_1,\dots,t_m\}\in P$ with $m\in \N$ and $0=t_0<t_1<\ldots <t_m$, we define the \textit{(maximal) mesh size} of $\pi$ by 
\[
 |\pi|_\infty := \max_{j=1,\dots,m} (t_j-t_{j-1}).
\]
Moreover, we set $|\{0\}|_\infty:=0$. Let $u_0\in \R^d$. In the following, we consider the limit of $\EE_\pi u_0$ when the mesh size of the partition $\pi\in P$ tends to zero. For this, we first remark that, for $h_1,h_2\geq 0$,
\begin{align*}
 \EE_{h_1+h_2}u_0&=\sup_{q\in \PP} S_\la(h_1+h_2)u_0=\sup_{q\in \PP} S_\la(h_1)S_\la(h_2)u_0\\
 &\leq \sup_{q\in \PP} S_\la(h_1)\EE_{h_2}u_0=\EE_{h_1}\EE_{h_2}u_0,
\end{align*}
which implies the inequality
\begin{equation}\label{monotone}
 \EE_{\pi_1}u_0\leq \EE_{\pi_2}u_0
\end{equation}
for $\pi_1,\pi_2\in P$ with $\pi_1\subset \pi_2$. 
The following lemma now states that $\SS(t)$, for $t\geq 0$, can be obtained by a pointwise monotone approximation with finite partitions letting the mesh size tend to zero.

\begin{lemma}\label{monlimit}
 Let $t\geq 0$ and $(\pi_n)_{n\in \N}\subset P_t$ with $\pi_n\subset \pi_{n+1}$ for all $n\in \N$ and $|\pi_n|_\infty\searrow 0$ as $n\to \infty$. Then, for all $u_0\in \R^d$,
 \[
  \EE_{\pi_n}u_0\nearrow \SS(t)u_0, \quad n\to \infty.
 \]
\end{lemma}

\begin{proof}
 Let $u_0\in \R^d$. For $t=0$ the statement is trivial. Therefore, assume that $t>0$ and let
  \begin{equation}\label{vdef}
  u_\infty:=\sup_{n\in \N} \EE_{\pi_n}u_0.
 \end{equation}
 As $\pi_n\subset \pi_{n+1}$ for all $n\in \N$, \eqref{monotone} implies that
 \[
  \EE_{\pi_n}u_0\nearrow u_\infty, \quad n\to \infty.
 \]
 Since $(\pi_n)_{n\in \N}\subset P_t$, we obtain that
 \[
  u_\infty \leq \SS(t)u_0.
 \]
 Next, we assume that the family $f$ is bounded. Let $\pi=\{t_0,t_1,\ldots, t_m\}\in P_t$ with $m\in \N$ and $0=t_0<t_1<\ldots <t_m=t$. Since $|\pi_n|_\infty\searrow 0$ as $n\to \infty$, we may w.l.o.g. assume that $|\pi_n|\geq m+1$ for all $n\in \N$. Again, since $|\pi_n|_\infty\searrow 0$ as $n\to \infty$, there exist $0=t_0^n<t_1^n<\ldots <t_m^n=t$ for all $n\in \N$ with $\pi_n':=\{t_0^n,t_1^n,\ldots, t_m^n\}\subset \pi_n$ and $t_i^n\to t_i$ as $n\to \infty$ for all $i\in \{1,\ldots, m\}$. Then, by Lemma \ref{Ehcont}, we have that
 \[
  \|\EE_\pi u_0-\EE_{\pi_n'}u_0\|_\infty\to 0, \quad n\to \infty
 \]
 and therefore,
 \[
  u_\infty\geq \EE_{\pi_n}u_0\geq \EE_{\pi_n'}u_0\geq \EE_\pi u_0-\|\EE_\pi u_0-\EE_{\pi_n'}u_0\|_\infty.
 \]
 Letting $n\to \infty$, we obtain that $u_\infty\geq \EE_\pi u_0$. Taking the supremum over all $\pi\in P_t$ yields the assertion for bounded $f$.\\
 Now, let $f$ again be (possibly) unbounded. It remains to show that $u_\infty\geq \SS(t)u_0$. By the previous step, we have that $u_\infty\geq u_{\infty,M}=\SS_M(t)$ for all $M\geq 0$, where $u_{\infty,M}$ is given by \eqref{vdef} but w.r.t. $(\PP_M,f_M)$ instead of $(\PP,f)$. Since $\SS_M(t)u_0\nearrow \SS(t)u_0$ as $M\to \infty$, we obtain that $u_\infty\geq \SS(t) u_0$, which ends the proof. 
\end{proof}

Choosing e.g. $\pi_n=\big\{\frac{kt}{2^n}\colon k\in \{0,\ldots ,2^n\}\big\}$ or $\pi_n=\big\{\frac{kt}{n!}\colon k\in \{0,\ldots, n!\}\big\}$ in Lemma \ref{monlimit}, we obtain the following corollaries.

\begin{corollary}\label{seq}
 For all $t>0$ there exists a sequence $(\pi_n)_{n\in \N}\subset P_t$ with
 \[
  \EE_{\pi_n}u_0\nearrow \SS(t)u_0
 \]
 as $n\to \infty$ for all $u_0\in \R^d$.
\end{corollary}

\begin{corollary}\label{seq1}
 For all $t\geq 0$ and $u_0\in \R^d$ we have that
 \[
  \SS(t)u_0=\sup_{n\in \N} \EE_{\frac{1}{n}}^n u_0=\lim_{n\to \infty}\EE_{2^{-n}}^{2^n} u_0.
 \]
\end{corollary}

\begin{proposition}\label{semchains}
 The family $\SS=(\SS(t))_{t\geq 0}$ defines a semigroup of convex kernels from $\R^d$ to $\R^d$. In particular, for all $s,t\geq 0$ we have the dynamic programming principle
\begin{equation}\label{111}
 \SS(s+t)=\SS(s)\SS(t).
\end{equation}
Moreover, the Nisio semigroup $\SS$ of $(\PP,f)$ coincides with the semigroup envelope of $(\PP,f)$ (cf. Definition \ref{nisiosg}).
\end{proposition}

\begin{proof}
 It remains to show the semigroup property \eqref{111}. Let $u_0\in \R^d$. If $s=0$ or $t=0$ the statement is trivial. Therefore, let $s,t>0$, $\pi_0\in P_{s+t}$ and $\pi:=\pi_0\cup \{s\}$. Then, $\pi\in P_{s+t}$ with $\pi_0\subset \pi$. Hence, by \eqref{monotone}, we obtain that
 \[
  \EE_{\pi_0} u_0\leq \EE_{\pi}u_0.
 \]
 Let $m\in \N$, $0=t_0<t_1<\ldots t_m=s+t$ with $\pi=\{t_0,\ldots, t_m\}$ and $i\in \{1,\ldots, m\}$ with $t_i=s$. Then,
 \[
  \pi_1:=\{t_0,\ldots, t_i\}\in P_s\quad \text{and} \quad \pi_2:=\{t_i-s,\ldots, t_m-s\}\in P_t\,
 \]
 with
 \[
  \EE_{\pi_1}=\EE_{t_1-t_0}\cdots \EE_{t_i-t_{i-1}} \quad \text{and}\quad \EE_{\pi_2}=\EE_{t_{i+1}-t_i}\cdots \EE_{t_m-t_{m-1}}.\qquad
  \]
 We thus see that
 \begin{align*}
  \EE_{\pi_0}u_0&\leq \EE_\pi u_0=\EE_{t_1-t_0}\cdots \EE_{t_m-t_{m-1}}u_0=\big(\EE_{t_1-t_0}\cdots \EE_{t_i-t_{i-1}}\big)\big(\EE_{t_{i+1}-t_i}\cdots \EE_{t_m-t_{m-1}}u_0\big)\\
  &=\EE_{\pi_1}\EE_{\pi_2}u_0\leq \EE_{\pi_1}\big(\SS(t)u_0\big)\leq \SS(s)\SS(t)u_0.
 \end{align*}
 Taking the supremum over all $\pi_0\in P_{s+t}$, it follows that $\SS(s+t)u_0\leq \SS(s)\SS(t)u_0$.\\
 Now, let $(\pi_n)_{n\in \N}\subset P_t$ with $\EE_{\pi_n}u_0\nearrow \SS(t)u_0$ as $n\to \infty$ (see Corollary \ref{seq}), and fix $\pi_0\in P_s$. Then, for all $n\in \N$,
 \[
  \pi_n':=\pi_0\cup \{s+\tau\colon \tau \in \pi_n\}\in P_{s+t}
 \]
 with $\EE_{\pi_n'}=\EE_{\pi_0}\EE_{\pi_n}$. Therefore,
 \[
  \EE_{\pi_0}\big(\SS(t)u_0\big)=\lim_{n\to \infty}\EE_{\pi_0}\EE_{\pi_n}u_0=\lim_{n\to \infty}\EE_{\pi_n'}u_0\leq \SS(s+t)u_0.
 \]
 Taking the supremum over all $\pi_0\in P_s$, yields that $\SS(s)\SS(t)u_0\leq \SS(s+t)u_0$. It remains to show that the family $\SS$ is the semigroup envelope of $(\PP,f)$. We have already shown that $\SS$ is a semigroup and, by definition, $\SS(t)u_0\geq S_q(t)u_0$ for all $t\geq 0$, $u_0\in \R^d$ and $q\in \PP$. Let $\big(\mathscr T(t)\big)_{t\geq 0}$ be a semigroup with $\mathscr T(t)u_0\geq S_q(t)u_0$ for all $t\geq 0$, $u_0\in \R^d$ and $q\in \PP$. Then,
 \[
  \EE_h u_0\leq \mathscr T(h)u_0 \quad \text{for all }h\geq 0\text{ and }u_0\in \R^d.
 \]
 Since $\big(\mathscr T(t)\big)_{t\geq 0}$ is a semigroup and $\EE_h$ is monotone for all $h\geq 0$, it follows that
 \[
 \EE_\pi u_0\leq \mathscr T(t)u_0\quad \text{for all }t\geq 0,\; \pi\in P_t\text{ and }u_0\in \R^d.
 \]
 Taking the supremum over all $\pi\in P_t$, it follows that $\SS(t)u_0\leq \mathscr T(t)u_0$ for all $t\geq 0$ and $u_0\in \R^d$.
\end{proof}

Proposition \ref{semchains} and \cite[Theorem 5.6]{MR3824739} imply the following corollary.

\begin{corollary}
There exists a convex Markov chain $\big(\Om,\FF,\EE,(X_t)_{t\geq 0}\big)$ such that
\[
 \big(\SS(t)u_0\big)_i=\EE_i\big(u_0(X_t)\big)
\]
for all $u_0\in \R^d$, $t\geq 0$ and $i\in \{1,\ldots, d\}$.
\end{corollary}

Restricting the time parameter of this process to $\N_0$, leads to a discrete-time Markov chain with transition operator $\SS(1)$ (cf. \cite[Example 5.3]{MR3824739}). It remains to show that the Nisio semigroup $\SS$ gives rise to the unique classical solution to the nonlinear ODE \eqref{odemain}.

\begin{remark}
 Assume that the set $\PP$ is bounded, i.e. $\sup_{q\in \PP} \|q\|<\infty$.
\begin{enumerate}
\item[a)] Since $\PP$ is bounded, it follows that $\QQ$ is Lipschitz continuous. Therefore, the Picard-Lindel\"of Theorem implies that, for every $u_0\in \R^d$, the initial value problem
  \begin{align}
   u'(t)&=\QQ u(t),\quad t\geq 0,\label{11}\\
   u(0)&=u_0,\notag
  \end{align}
has a unique solution $u\in C^1\big([0,\infty);\R^d\big)$. We will show that this solution $u$ is given by $u(t)=\SS(t)u_0$ for all $t\geq 0$. That is, the unique solution of the ODE \eqref{11} is given by the Nisio semigroup. 
\item[b)] Since $\PP$ is bounded, the mapping
\[
 \mathfrak{q}\colon \R^d\to \R^d,\quad u\mapsto \sup_{q\in \PP} qu
\]
is well-defined.
\end{enumerate}
\end{remark}

The following key estimate and its proof are a straightforward adaption of the proof of \cite[Proposition 5]{MR3290231} to our setup. Recall that, by Remark \ref{uniclass1}, the boundedness of the family $f$ implies the boundedness of the set $\PP$.
\begin{lemma}\label{intsemi}
 Assume that the family $f$ is bounded. Then,
 \[
  \SS(t) u_0-u_0\leq \int_0^t \Si(s)\QQ u_0\, {\rm d}s
 \]
 for all $u_0\in \R^d$ and $t\geq 0$. Here, $\big(\Si(t)\big)_{t\geq 0}$ is the Nisio semigroup w.r.t. the sublinear $Q$-operator $\mathfrak{q}$ from the previous remark, or more precisely, the Nisio semigroup w.r.t. $(\PP,f)$, where $f_q=0$ for all $q\in \PP$.
\end{lemma}

\begin{proof}
Let $u_0\in \R^d$ and $h>0$. Then, by \eqref{duhamel}, we have that
\[
 S_q(h)u_0-u_0=\int_0^h e^{sq}\big(qu_0+f_q\big)\, {\rm d}s\leq \int_0^h \Si(s)\QQ u_0\, {\rm d}s.
\]
Notice that, by Lemma \ref{unicont}, the mapping $[0,\infty)\to \R^d,\; t\mapsto \Si(t) v_0$ is continuous and therefore locally integrable for all $v_0\in \R^d$. Hence, for all $\tau\geq 0$,
\begin{equation}\label{intsemiproof0}
\EE_h u_0-u_0\leq \int_0^h \Si(s)\QQ u_0\, {\rm d}s=\int_\tau^{\tau+h} \Sigma(s-\tau)\QQ u_0\, {\rm d}s.
\end{equation}
Next, we show that
 \begin{equation}\label{intsemiproof}
  \EE_\pi u_0-u_0\leq \int_0^{\max \pi} \Si(s)\QQ u_0 \, {\rm d}s
 \end{equation}
 for all $\pi\in P$ by an induction on $m=| \pi|$, where $|\pi |$ denotes the cardinality of $\pi$. If $m=1$, i.e. if $\pi=\{0\}$, the statement is trivial. Hence, assume that
  \[
  \EE_{\pi'} u_0-u_0\leq \int_0^{\max \pi'} \Si(s)\QQ u_0 \, {\rm d}s
 \]
 for all $\pi'\in P$ with $|\pi'|=m$ for some $m\in \N$. Let $\pi=\{t_0,t_1,\ldots, t_m\}\in P$ with $0=t_0<t_1<\ldots <t_m$ and $\pi':=\pi\sm \{t_m\}$. Then, we obtain that
 \begin{align*}
  \EE_\pi u_0 -\EE_{\pi'}u_0&\leq \Si(t_{m-1})\big(\EE_{t_m-t_{m-1}}u_0-u_0\big)\leq \Si(t_{m-1})\bigg(\int_{t_{m-1}}^{t_m} \Si(s-t_{m-1})\QQ u_0\, {\rm d}s\bigg)\\
  &\leq\int_{t_{m-1}}^{t_m}\Si(s)\QQ u_0\, {\rm d}s,
 \end{align*}
 where, in the second inequality, we used \eqref{intsemiproof0} with $h=t_m-t_{m-1}$ and $\tau=t_{m-1}$, and, in the last inequality, we used the sublinearity of $\Si(t)$. Using the induction hypothesis, we thus see that
 \begin{align*}
 \EE_\pi u_0 -u_0&= \big(\EE_\pi u_0 -\EE_{\pi'}u_0\big)+\big(\EE_{\pi'} u_0-u_0\big)\leq \int_{t_{m-1}}^{t_m}\Si(s)\QQ u_0\, {\rm d}s+\int_{0}^{t_{m-1}}\Si(s)\QQ u_0\, {\rm d}s\\
 &=\int_0^{\max \pi}\Si(s)\QQ u_0\, {\rm d}s.
 \end{align*}
 \item By \eqref{intsemiproof}, it follows that
 \[
  \EE_\pi u_0-u_0\leq \int_0^t\Si(s)\QQ u_0\, {\rm d}s
 \]
 for all $\pi\in P_t$. Taking the supremum over all $\pi\in P_t$ we obtain the assertion.
\end{proof}
 
The following proposition states that the Nisio semigroup $(\SS(t))_{t\geq 0}$ is differentiable at zero if the family $f$ is bounded.
\begin{proposition}\label{diffzero}
 Assume that $f$ is bounded. Then, for all $u_0\in \R^d$,
 \[
  \bigg\|\frac{\SS(h)u_0-u_0}{h}-\QQ u_0\bigg\|_\infty\to 0,\quad h\searrow 0.
 \]
\end{proposition}

\begin{proof}
 Since $f$ is bounded, it follows that $\PP$ is bounded (see Remark \ref{uniclass1}). Let $\ep>0$ and $u_0\in \R^d$. Using Lemma \ref{unicont}, the boundedness of $\PP$ and \eqref{bla1}, there exists some $h_0>0$ such that, for all $0<h\leq h_0$,
 \begin{align*}
  \big\|e^{hq} \big(q u_0+f_q\big)-\big(q u_0+f_q\big)\big\|_\infty&\leq \|e^{hq}-I_d\|\cdot \|q u_0+f_q\|_\infty\\
  &\leq \big(e^{\|q\|h}-1\big)\|q u_0+f_q\|_\infty\leq \ep, 
 \end{align*}
 for all $q\in \PP$, and
 \[
  \Si(h)\QQ u_0-\QQ u_0\leq\ep.
 \]
 Let $0<h\leq h_0$. Then,
 \[
  \SS(h)u_0-u_0 \geq S_q(h)u_0-u_0 =\int_0^h e^{tq} \big(q u_0+f_q\big)\, {\rm d}s \geq \big(q u_0+f_q-\ep \big)h
 \]
 for all $q\in \PP$. Dividing by $h$ and taking the supremum over all $q\in \PP$, it follows that
 \begin{equation}\label{gen12134}
  \frac{\SS(h)u_0-u_0}{h}\geq \QQ u_0-\ep.
 \end{equation}
 Moreover, by Lemma \ref{intsemi},
 \[
  \SS(h)u_0-u_0-h\QQ u_0\leq \int_0^h \Si(s)\QQ u_0\, {\rm d}s-h\QQ u_0=\int_0^h \big(\Si(s)\QQ u_0-\QQ u_0\big)\, {\rm d}s\leq h\ep.
 \]
 Dividing again by $h>0$ yields
 \[
  \frac{\SS(h)u_0-u_0}{h}-\QQ u_0\leq  \ep,
 \]
 which, together with \eqref{gen12134}, implies that
 \[
  \bigg\|\frac{\SS(h)u_0-u_0}{h}-\QQ u_0\bigg\|_\infty\leq \ep.
 \]
\end{proof}


\begin{corollary}\label{cormain}
 Let $f$ be bounded, $u_0\in \R^d$ and $u(t):=\SS(t)u_0$ for $t\geq 0$. Then, $u\in C^1\big([0,\infty);\R^d\big)$ is the unique classical solution to the ODE
 \[
   u'(t)=\QQ u(t),\quad t\geq 0
  \]
 with $u(0)=u_0$.
\end{corollary}

\begin{proof}
 Let $u_0\in \R^d$ and $t\geq 0$. Then, by Proposition \ref{diffzero},
 \[
  \lim_{h\searrow 0}\frac{\SS(t+h)u_0-\SS(t)u_0}{h}\lim_{h\searrow 0}\frac{\SS(h)\SS(t)u_0-\SS(t)u_0}{h}=\QQ\SS(t)u_0.
 \]
 This shows that the map $u\colon [0,\infty)\to \R^d,\; t\mapsto \SS(t)u_0$ is continuous (see Lemma \ref{unicont}) and right differentiable with continuous right derivative $$[0,\infty)\to \R^d,\quad t\mapsto \QQ\SS(t)u_0,$$ where we used that the fact that $\QQ\colon \R^d\to \R^d$ is convex and thus continuous. Therefore, $u$ is continuously differentiable with $u'(t)=\QQ u(t)$, for all $t\geq 0$, and $u(0)=u_0$. The Picard-Lindel\"of Theorem together with the local Lipschitz continuity of the convex map $\QQ\colon \R^d\to \R^d$ implies the uniqueness of $u$.
\end{proof}

\begin{corollary}\label{unico}
 Let $f$ be bounded. Then, there exists some constant $L>0$ such that
 \[
  \|\SS(t)u_0-u_0\|_\infty\leq Lt\|u_0\|_\infty
 \]
for all $t\geq 0$ and $u_0\in \R^d$.
\end{corollary}

\begin{proof}
 Since $f$ is bounded, we have that $\PP$ is bounded and therefore $\QQ$ is Lipschitz continuous with Lipschitz constant $L:=\sup_{q\in \PP}\|q\|$. For all $u_0\in \R^d$ we thus obtain that
 \[
  \|\SS(t)u_0-u_0\|_\infty\leq \int_0^t\|\QQ\SS(s)u_0\|_\infty\, {\rm d}s\leq \int_0^tL\|\SS(s)u_0\|_\infty\, {\rm d}s\leq Lt \|u_0\|_\infty.
 \]
\end{proof}

Finally, in order to end the proof of Theorem \ref{main}, we have to extend Corollary \ref{cormain} to the unbounded case. We start with the following remark, which is the key observation in order to finish the proof of Theorem \ref{main}.
\begin{remark}\label{uniclass}
Let $\PP^*:=\{q\in \R^{d\times d}\, |\, \QQ^*(q)<\infty\}$ and $f^*_q:=-\QQ^*(q)$ for all $q\in \PP^*$, where $\QQ^*$ is the conjugate function of $\QQ$ (cf. Remark \ref{conjfunct}). For all $M\geq 0$, let $(\PP_M^*, f_M^*)$ and $\QQ_M^*$ be as in Remark \ref{uniclass1} with $\PP$ being replaced by $\PP^*$. Moreover, let $\big(\SS^*_M(t)\big)_{t\geq 0}$ be the Nisio semigroup w.r.t. $(\PP_M^*,f_M^*)$ for $M\geq 0$. As
\[
 \bigcup_{M\geq 0}\PP_M^*=\PP^*,
\]
it follows that $\SS_M^*(t)\nearrow \SS^*(t)$ as $M\to \infty$ for all $t\geq 0$, where $(\SS^*(t))_{t\geq 0}$ is the Nisio semigroup w.r.t.~$(\PP^*,f^*)$. Let $R>0$ be fixed. Then, there exists some $M_0\geq 0$ such that $\QQ u=\QQ_{M_0}^*u$ for all $u\in \R^d$ with $\|u\|_\infty\leq R$, by choice of $\PP^*$ and $f^*$. Let $u_0\in \R^d$ with $\|u_0\|_\infty\leq R$. Then, it follows that $\|\SS_M^*(t)u_0\|_\infty\leq R$ for all $t\geq 0$ and $M\geq 0$, which implies that $\SS_M^*(t)u_0=\SS_{M_0}^*(t)u_0$ for all $t\geq 0$ and $M\geq M_0$ by the uniqueness obtained from the Picard-Lindel\"of Theorem. In particular, $\SS^*(t)u_0=\SS_{M_0}^*(t)u_0$ for all $t\geq 0$, which shows that the nonlinear ODE \eqref{odemain} has a unique classical solution $u^*\in C^1\big([0,\infty);\R^d\big)$ with $u^*(0)=u_0$. This solution is given by $u^*(t)=\SS^*(t)u_0$ for all $t\geq 0$. By Corollary \ref{unico}, we thus get that $\SS^*(t)\to I$ as $t\searrow 0$ uniformly on compact sets. 
\end{remark}

We are now able to finish the proof of Theorem \ref{main}.

\begin{proposition}\label{endmain}
Let $u_0\in \R^d$. Then, $u\colon [0,\infty)\to \R^d,\; t\mapsto \SS(t)u_0$ is the unique classical solution $u\in C^1\big([0,\infty);\R^d\big)$ to the initial value problem
 \begin{align*}
  u'(t)&=\QQ u(t),\quad t\geq 0,\\
  u(0)&=u_0.
 \end{align*}
 Moreover, the Nisio semigroup $(\SS(t))_{t\geq 0}$ is uniformly continuous (see Definition \ref{semigroup}).
\end{proposition}

\begin{proof}
 By Remark \ref{uniclass}, the initial value problem
  \begin{align*}
  u'(t)&=\QQ u(t),\quad t\geq 0,\\
  u(0)&=u_0.
 \end{align*}
 has a unique classical solution $u^*\in C^1\big([0,\infty);\R^d\big)$, which is given by
 \[
 u^*(t):=\SS^*(t)u_0\quad \text{for all }t\geq 0.
 \]
 We show that $u^*(t)=\SS(t)u_0$, for all $t\geq 0$. Let $R:=\|u_0\|_\infty$. For all $M\geq 0$, let $(\PP_M,f_M)$, $\QQ_M$ and $\SS_M=\big(\SS_M(t)\big)_{t\geq 0}$ be as in Remark \ref{uniclass1}. Let $\ep>0$. Since $\QQ\colon \R^d\to \R^d$ is convex, it is locally Lipschitz. Hence, by Dini's lemma, there exists some $M_0\geq 0$ such that
\[
 \|\QQ v_0-\QQ_{M_0}v_0\|_\infty\leq \ep
\]
for all $v_0\in \R^d$ with $\|v\|_\infty\leq R$. Further, there exists some constant $L>0$ such that
\[
 \|\QQ v_1-\QQ v_2\|_\infty\leq L\|v_1-v_2\|_\infty
\]
for all $v_1,v_2\in \R^d$ with $\|v_1\|_\infty\leq R$ and $\|v_2\|_\infty\leq R$. Since $\|u^*(t)\|_\infty\leq R$ and $\|\SS_M(t)u_0\|_\infty\leq R$ for all $M\geq 0$ and $t\geq 0$, we obtain that
\begin{align*}
 \|\SS_M(t)u_0-u^*(t)\|_\infty&=\bigg\|\int_0^t \QQ_M\SS_M(s)u_0-\QQ u^*(s)\, {\rm d}s\bigg\|_\infty\\
 &\leq \int_0^t \|\QQ_M\SS_M(s)u_0-\QQ u^*(s)\|_\infty\, {\rm d}s\\
 &\leq  \int_0^t \big(\|\QQ\SS_M(s)u_0-\QQ u^*(s)\|_\infty+\ep\big)\, {\rm d}s\\
 &\leq  \int_0^t L\|\SS_M(s)u_0- u^*(s)\|_\infty+\ep\, {\rm d}s
\end{align*}
for all $t\geq 0$ and $M\geq M_0$. By Gronwall's lemma, we thus get that
\[
 \|\SS_M(t)u_0-u^*(t)\|_\infty\leq \ep t e^{Lt}
\]
for all $t\geq 0$ and $M\geq M_0$, showing that $\SS_M(t)u_0\to u^*(t)$ as $M\to \infty$ for all $t\geq 0$. However, since $\SS_M(t)u_0\nearrow \SS(t)u_0$ as $M\to \infty$ for all $t\geq 0$, we obtain that $u^*(t)=\SS(t)u_0$. This shows that $\SS(t)=\SS^*(t)$ for all $t\geq 0$, which, together with Remark \ref{uniclass}, implies that $\SS(t)=\SS^*(t)\to I$ uniformly on compact sets as $h\searrow 0$. This ends the proof of this proposition and also the proof of Theorem \ref{main}.
\end{proof}

We conclude this section with the following remark, where we derive a dual representation of the semigroup envelope.

\begin{remark}\label{rem:dualrep}
We will now derive a dual representation of the semigroup envelope by viewing the semigroup envelope as the cost functional of an optimal control problem, where, roughly speaking, ``nature'' tries to control the system into the worst possible scenario (using contols within the set $\PP$). For $q=(q^1,\ldots q^d)\in \PP^d$ and $t\geq 0$, let $S_q(t)\in \R^{d\times d}$ be given by
\begin{equation}\label{sqi}
\big(S_q(t) u_0\big)_i:=\big(S_{q^i}(t)u_0\big)_i
\end{equation}
for all $u_0\in \R^d$ and $i\in \{1,\ldots, d\}$. That is, $S_q(t)$ is the matrix whose $i$-th row is the $i$-th row of $S_{q^i}(t)$ for all $i\in \{1,\ldots, d\}$. Here, the interpretation is that, in every state $i\in \{1,\ldots, d\}$, ``nature'' is allowed to choose a different model $q\in \PP$. We now add a dynamic component, and define
\[
 Q_t:=\bigg\{ (q_k,h_k)_{k=1,\ldots, m}\in \big(\PP^d\times [0,t]\big)^m\, \bigg|\,m\in \N,\; \sum_{k=1}^m h_k=t\bigg\}.
\]
Roughly speaking, $Q_t$ corresponds to the set of all (space-time discrete) admissible controls for the control set $\PP$. For an admissible control $\theta=(q_k,h_k)_{k=1,\ldots, m}\in Q_t$ with $m\in \N$ and $u_0\in \R^d$, we then define
\[
 S_\theta u_0:=S_{q_1}(h_1)\cdots S_{q_m}(h_m)u_0,
\]
 where $S_{q_k}(h_k)$ is defined as in \eqref{sqi} for $k=1,\ldots, m$. Then, for all $u_0\in \R^d$,
\begin{equation}\label{primaldual}
 \SS(t)u_0=\sup_{\pi\in P_t} \EE_\pi u_0=\sup_{\theta\in Q_t}S_\theta u_0.
\end{equation}
In fact, by definition of $Q_t$, it follows that $S_q(t)u_0\leq \sup_{\theta\in Q_t}S_\theta u_0\leq \SS(t)u_0$ for all $q\in \PP$, $t\geq 0$ and $u_0\in \R^d$. On the other hand, one readily verifies that $\mathscr T(t)u_0:= \sup_{\theta\in Q_t}S_\theta u_0$, for $t\geq 0$ and $u_0\in \R^d$, gives rise to a semigroup $\big(\mathscr T(t)\big)_{t\geq 0}$. Since $\big(\SS(t)\big)_{t\geq 0}$ is the semigroup envelope of $(\PP,f)$, it follows that $\mathscr T(t)=\SS(t)$ for all $t\geq 0$.
\end{remark}

\section{Computation of price bounds under model uncertainty}\label{sec:numex}
In this section, we demonstrate how price bounds for European contingent claims under uncertainty can be computed numerically in certain scenarios, firstly, via the explicit primal/dual description \eqref{primaldual} of the semigroup envelope and, secondly, by solving the pricing ODE \eqref{odemain}. Throughout, we consider two $Q$-matrices $q_0\in \R^{d\times d}$ and $q\in \R^{d\times d}$ and, for $\lambda_l,\lambda_h\in \R$ with $\lambda_l\leq \lambda_h$, the interval $[\la_l,\lambda_h]$. Then, we consider the $Q$-operator $\QQ\colon \R^d\to \R^d$ given by
\[
 \QQ u_0:=q_0 u_0+\max_{\lambda\in [\lambda_l,\lambda_h]} \lambda qu_0\quad \text{for all }u_0\in \R^d.
\]
Then, $\QQ$ is sublinear and has the dual representation $\big(\{q_0+\lambda_l q,q_0+\lambda_h q\},(0,0)\big)$. Choosing the latter dual representation, we may compute $\QQ$ and $\EE_h$, for $h\geq 0$, via
\begin{equation}\label{eq:genuncert}
 \QQ u_0=\max_{\lambda=\lambda_l,\lambda_h} \lambda q_0 u_0+qu_0\quad \text{for all }u_0\in \R^d.
\end{equation}
and
\begin{equation}\label{ex:eh}
 \EE_h u_0=\max_{\lambda=\lambda_l,\lambda_h} e^{h(q_0+\lambda q)}u_0\quad \text{for all }u_0\in \R^d.
\end{equation}
In the sequel, we use \eqref{eq:genuncert} and \eqref{ex:eh} in order to compute upper bounds for prices of European contingent claims under uncertainy. Replacing the maximum by a minimum in \eqref{eq:genuncert} and \eqref{ex:eh}, we obtain lower bounds for the prices. In the examples, we consider, for suitable $\delta>0$, the rate matrix
  \begin{equation}\label{eq:mata}
   a:=\frac{1}{\delta^2}\left(\begin{array}{ccccccc}
    -1 & 1 &0& 0& 0& \cdots &0\\
    1 & -2 & 1& 0&0 & \cdots & 0\\
    0 & 1& -2& 1 & 0&\cdots & 0\\
    \vdots&\ddots & \ddots& \ddots& \ddots & \ddots &\vdots \\
   0&\cdots &0& 1&-2&1&0 \\
   0  &\cdots &0& 0& 1&-2&1\\
   0 &\cdots &0&0 & 0&1 &-1
   \end{array}\right),
  \end{equation}
 which is a discretization of the second space derivative with Neumann boundary conditions, and the rate matrix
  \begin{equation}\label{eq:matb}
   b:=\frac{1}{\delta}\left(\begin{array}{cccccc}
    -1 & 1 & 0&0& \cdots &0\\
    0 & -1 & 1& 0 & \cdots & 0\\
    \vdots& \ddots& \ddots& \ddots & \ddots &\vdots \\
   0&\cdots & 0&-1&1&0 \\
   0  &\cdots & 0& 0&-1&1\\
   0 &\cdots &0 & 0& 0 & 0
   \end{array}\right)
  \end{equation}
 as a discretization of the first space derivative. Then, the rate matrix
 \[
  \frac{\sigma^2}{2} a+\mu b,\quad \text{for }\sigma>0\text{ and }\mu\in \R,  
 \]
is a finite-difference discretization of $\frac{\sigma^2}{2}\partial_{xx}+\mu\partial_x$, which is the generator of a Brownian Motion with volatility $\sigma$ and drift $\mu$.\\

We start with an example, where we demonstrate how the semigroup envelope, and thus price bounds for Euproean contingent claims under model uncertainty, can be computed by solving the nonlinear pricing ODE \eqref{odemain}.

\begin{example}\label{ex:numode}
 In this example, we compute the semigroup envelope $\big(\SS(t)\big)_{t\geq 0}$ by solving the ODE $u'=\QQ u$ $u(0)=u_0\in \R^d$ with the explicit Euler method. The latter could be replaced by any other Runge-Kutta method. We consider the case, where, $d=101$, $\delta=\frac{1}{10}$. The state space is $S=\{i\delta\, |\, i\in \{0,\ldots, 100\}\}$, which as a discretization of the interval $[0,10]$, the maturity is $t=1$, and we choose $1000$ time steps in the explicit Euler method. We consider the following two examples.
 \begin{enumerate}
   \item[a)]  Let $\QQ$ be given by \eqref{eq:genuncert} with $q_0:=a$, $q:=b$, $\lambda_l:=-1$ and $\lambda_h:=1$, i.e. we consider the case of an uncertain drift parameter in the interval $[-1,1]$. We price a butterfly spread, which is given by
   \begin{equation}\label{eq:butterfly}
  \qquad u_0(x)=\big(L-K-|x-L|\big)^+,\quad \text{for }x=i\delta \text{ and }i\in \{1,\ldots, 100\},
   \end{equation}
   with $K=4$ and $L=5$.
   \begin{figure}
    \centering
     \includegraphics[width=14cm]{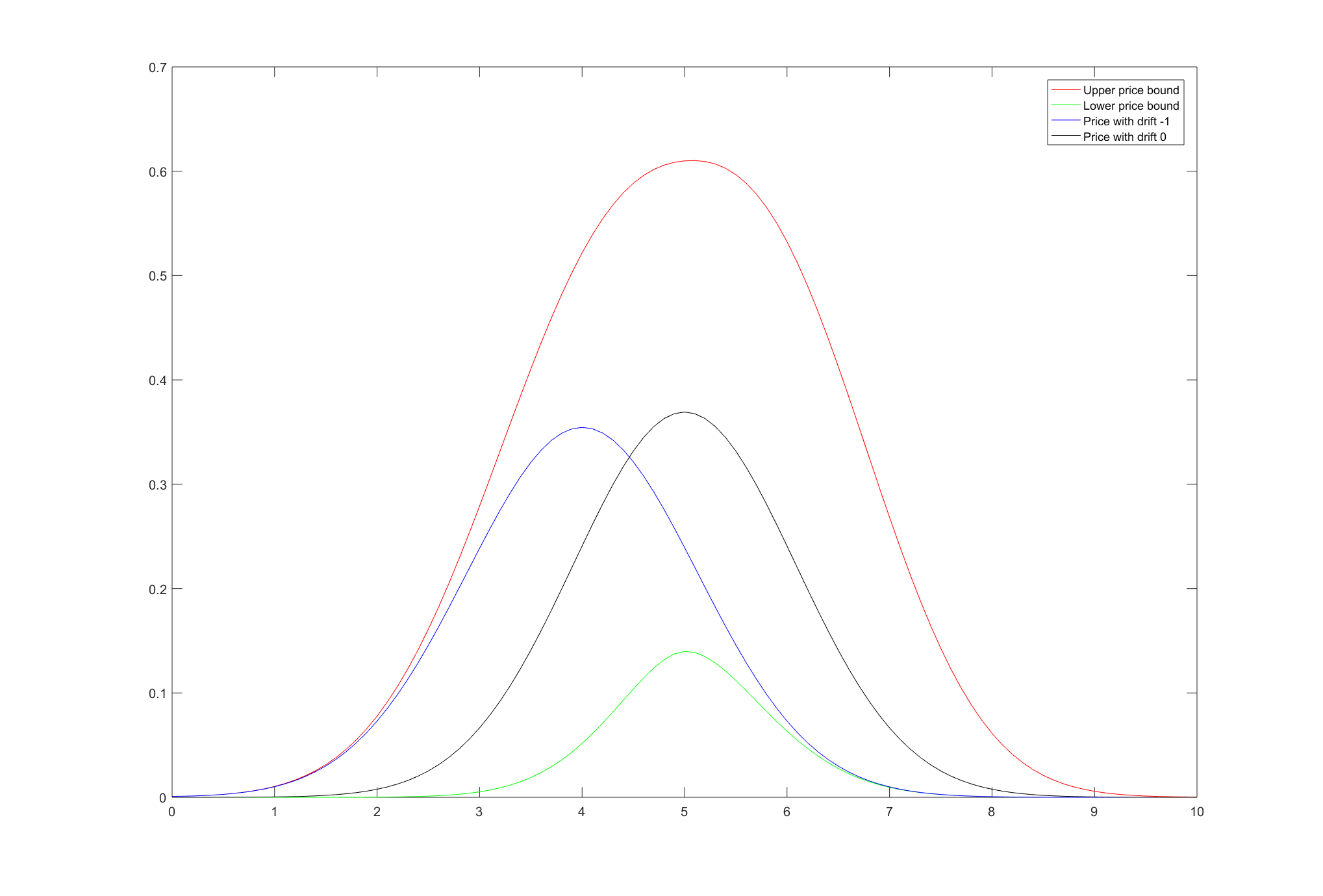}
     \caption{Upper and lower price bounds for a butterfly spread \eqref{eq:butterfly} with $K=4$ and $L=5$ under drift uncertainty depending on the current price in red and green, respectively. In blue and black, we see the value of the butterfly in the Bachelier model with drift $-1$ and $0$, respectively.}
   \end{figure}
   In Figure 1, we depict the upper and lower price bounds as well as the prices corresponding to the Bachelier model with drift $-1$ and $0$ in blue and black, respectively.
 \item[b)] Now, let $q_0:=0$, $q:=a$, $\lambda_l:=0.5$ and $\lambda_h:=1.5$ in \eqref{eq:genuncert}. That is, we consider the case of an uncertain volatility in the interval $[0.5,1.5]$. We price a bull spread
 \begin{equation}\label{eq:bullspread}
 \quad \qquad u_0(x)=\min\big\{(x-K)^+,L-K\big\},\quad \text{for }x=i\delta\text{ and }i\in \{1,\ldots, 100\},
 \end{equation}
 with $K=4$ and $L=5$.
    \begin{figure}
    \centering
     \includegraphics[width=14cm]{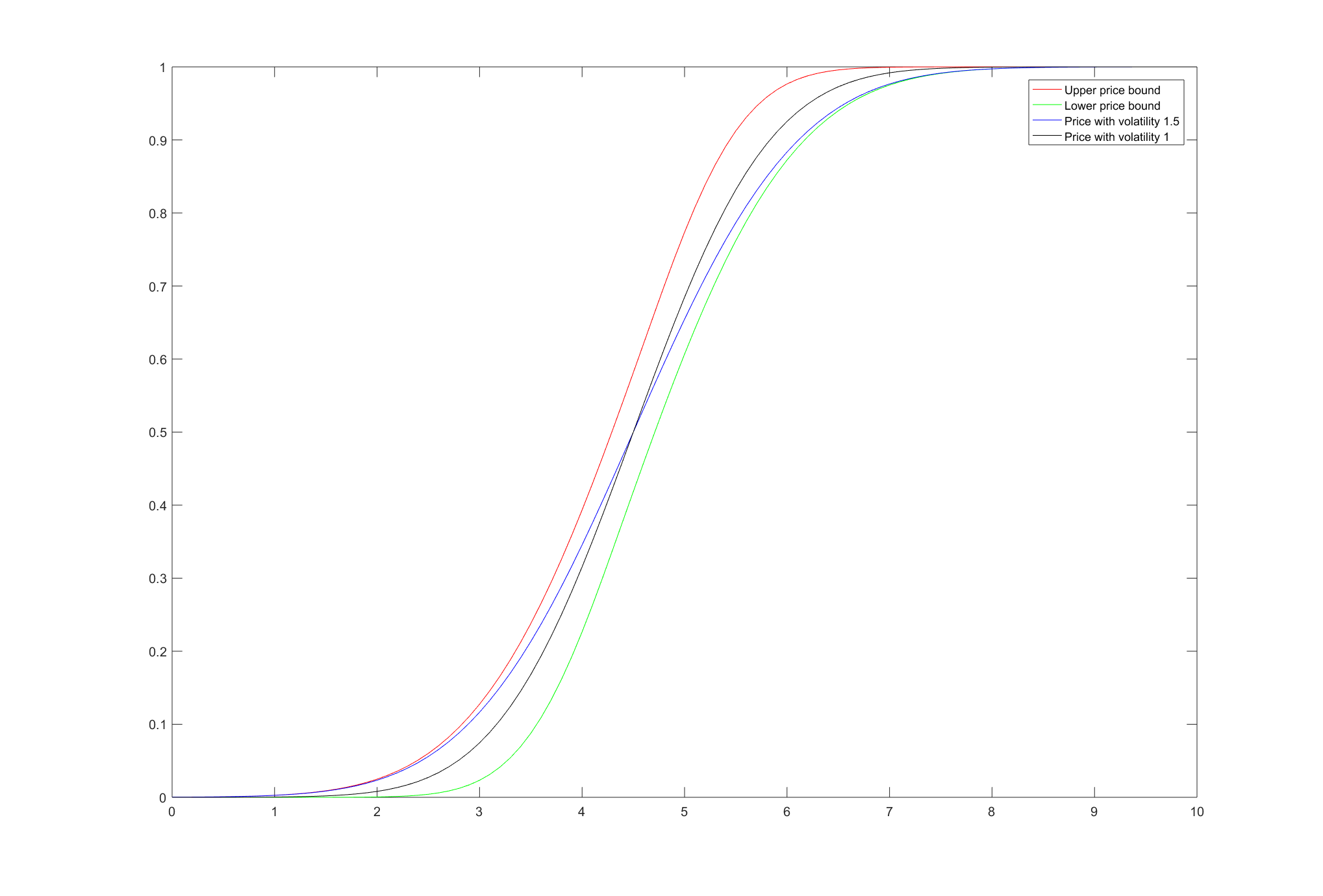}
     \caption{Upper and lower price bounds for a bull spread \eqref{eq:bullspread} with $K=4$ and $L=5$ under volatility uncertainty depending on the current price in red and green, respectively. In black and blue, we see the value of the butterfly in the Bachelier model with drift $1$ and $1.5$, respectively.}
   \end{figure}
   In Figure 1, we see the upper and lower price bounds as well as the prices corresponding to the Bachelier model with volatility $1$ and $1.5$ in black and blue, respectively. 
 \end{enumerate}
\end{example}

The following example illustrates how the primal/dual representation of the semigroup envelope can be used to compute bounds for prices of European contingent claims under model uncertainty.

\begin{example}\label{ex:numdualrep}
For a fixed maturity $t\geq 0$, we consider the partitions
$$\pi_n:=\{k2^{-n}t\, |\, k=0,\ldots, 2^n\},\quad\text{for }n\in \N_0,$$
of the time interval $[0,t]$. We are then able to approximate the upper bound for prices of European contingent claims under uncertainy with maturity $t=1$ by computing, for $n\in \N_0$ sufficiently large,
\begin{equation}\label{eq:compmeth1}
 \underbrace{\EE_{2^{-n}t}\cdots \EE_{2^{-n}t}}_{2^n-\text{times}}u
\end{equation}
with $\EE_h$ given by \eqref{ex:eh} for $h\geq 0$. The fundamental system $e^{h(q_0+\lambda q)}$, for $\lambda=\lambda_h,\lambda_l$, appearing in \eqref{ex:eh} can either be computed via the Jordan decomposition of $q_0+\lambda q$, by the approximation
\begin{equation}\label{eq:approxexp}
 \big(I+\tfrac{h}{k}(q_0+\lambda q)\big)^k
\end{equation}
with $k\in \N_0$ sufficiently large or by numerically solving the matrix-valued ODE $$U'=(q_0+\lambda q)U\quad \text{with}\quad U(0)=I,$$
where $I=I_d$ is the $d\times d$-identity matrix. We illustrate the approximation of the semigroup envelope via \eqref{eq:compmeth1} in the following two examples, where $a$ and $b$ are given by \eqref{eq:mata} and \eqref{eq:matb}. Again, we consider the case, where, $d=101$, $\delta=\frac{1}{10}$ and the maturity is $t=1$. In both examples, we choose $n=10$, i.e. we consider the partition $\pi_{10}$ with $t=1$, and use \eqref{eq:approxexp} with $k=10$ for the computation of $e^{h(q_0+\lambda q)}$ for $\lambda=\lambda_h,\lambda_l$.
 \begin{enumerate}
  \item[a)] As in Example \ref{ex:numode} a), let $q_0:=a$, $q:=b$, $\lambda_l:=-1$ and $\lambda_h:=1$. Again, we compute the price of a butterfly spread, which is given by \eqref{eq:butterfly} with $K=4$ and $L=5$.
     \begin{figure}
    \centering
     \includegraphics[width=14cm]{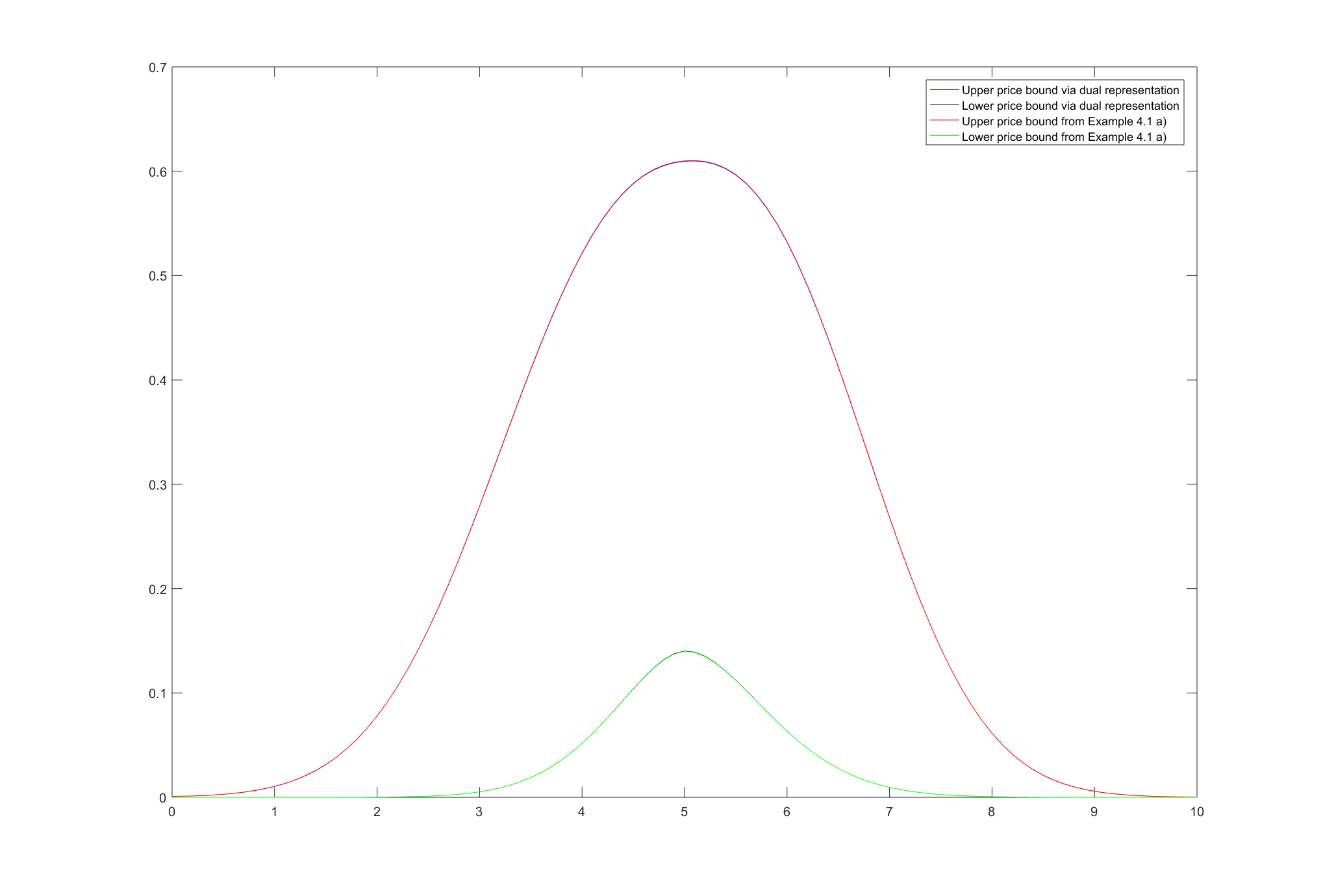}
     \caption{Upper and lower price bounds for a butterfly spread \eqref{eq:butterfly} with $K=4$ and $L=5$ under drift uncertainty from Example \ref{ex:numode} a) in red and green, respectively. In blue and black, the upper and lower price bounds, computed via \eqref{eq:compmeth1}, respectively.}
   \end{figure}
   In Figure 3, we see the upper and lower price curves from the previous example as well as the price bounds computed in this example. We observe that the price bounds match very well.
  \item[b)] We consider the case of an uncertain volatility parameter from Example \ref{ex:numode} b), i.e. let $q_0:=0$, $q:=a$, $\lambda_l:=0.5$ and $\lambda_h:=1.5$. As in Example \ref{ex:numode} b), we price a bull spread given by \eqref{eq:bullspread} with $K=4$ and $L=5$.
     \begin{figure}
    \centering
     \includegraphics[width=14cm]{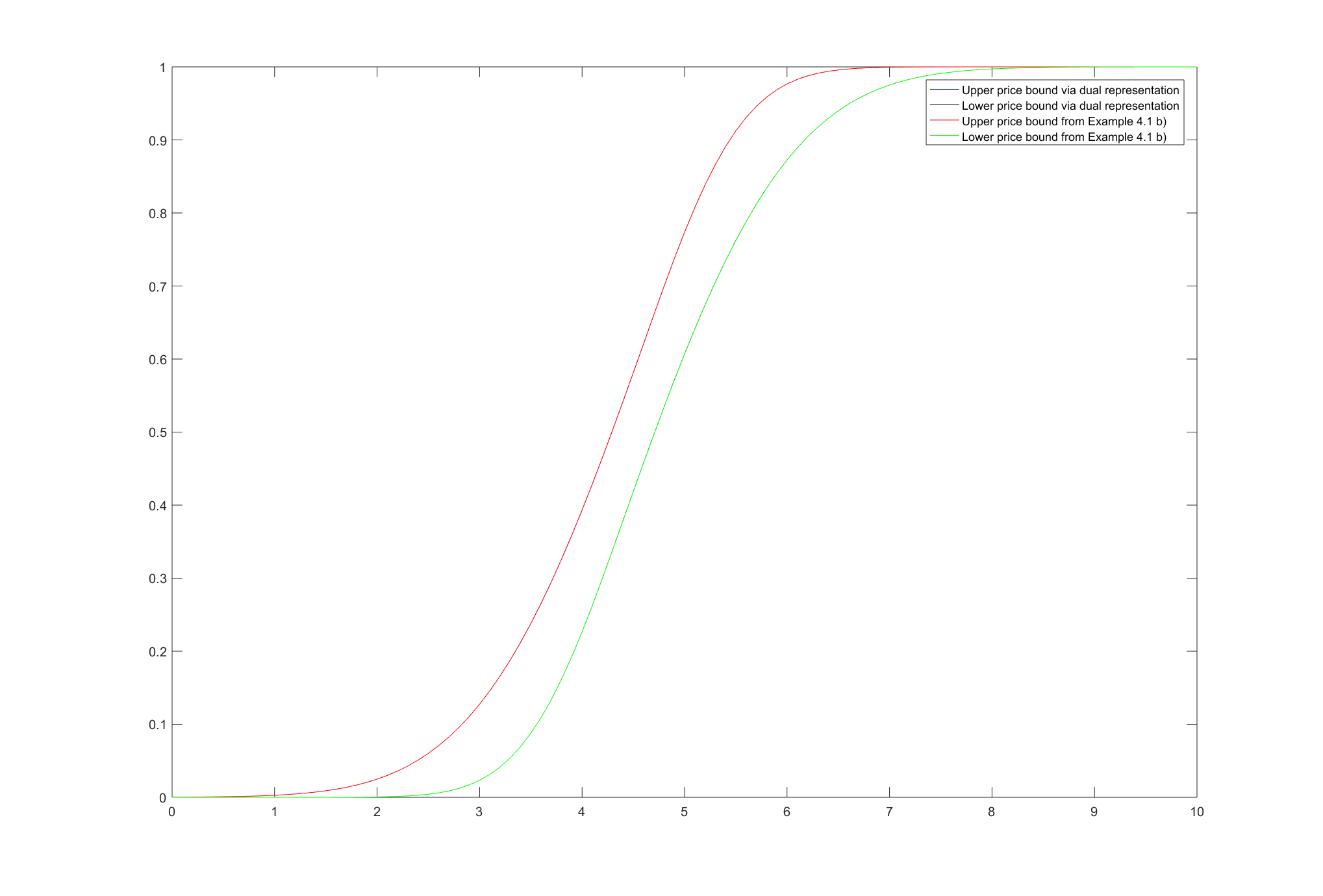}
     \caption{Upper and lower price bounds for a bull spread \eqref{eq:bullspread} with $K=4$ and $L=5$ under volatility uncertainty from Example \ref{ex:numode} b) in red and green, respectively. In blue and black, the upper and lower price bounds, computed via \eqref{eq:compmeth1}, respectively.}
   \end{figure}
   In Figure 4, we again depict the upper and lower price bounds from the previous example and this example. As in part a), we observe that the price bounds perfectly match.
 \end{enumerate}
\end{example}

\bibliographystyle{abbrv}

\end{document}